\documentclass[12pt]{amsart}
\usepackage[margin=1.25in]{geometry}
\usepackage{latexsym}
\usepackage{amsmath}
\usepackage{amssymb}
\usepackage{amsthm}
\usepackage{stmaryrd}
\usepackage{amscd}
\usepackage{enumerate} 
\usepackage{amssymb} 
\usepackage{mathrsfs}
\usepackage[all]{xy}
\usepackage{color}
\usepackage{graphicx}
\usepackage{mathtools}
\usepackage{comment}
\usepackage{multirow}
\usepackage{aliascnt}

\makeatletter
  
  \@addtoreset{equation}{thm}
  \makeatother

\title{Ultra-test ideals in rings with finitely generated anti-canonical algebras}

\author{Tatsuki Yamaguchi}
\address{Department of Mathematics, School of Science, Institute of Science Tokyo}
\email{yamaguchi.t.bp@m.titech.ac.jp}

\def\ge{\geqslant}
\def\le{\leqslant}
\def\phi{\varphi}
\def\epsilon{\varepsilon}
\def\tilde{\widetilde}

\def\mapsto{\longmapsto}

\def\Hom{\operatorname{Hom}}
\def\Spec{\operatorname{Spec}}
\def\Proj{\operatorname{Proj}}

\def\Div{\operatorname{div}}

\def\ulim{\operatorname{ulim}}
\def\Ann{\operatorname{Ann}}
\def\adj{\operatorname{adj}}

\def\m{{\mathfrak m}}
\def\n{{\mathfrak n}}
\def\ba{{\mathfrak a}}
\newcommand{\F}{\mathcal{F}}
\newcommand{\N}{\mathbb{N}}
\newcommand{\Q}{\mathbb{Q}} 
\newcommand{\C}{\mathbb{C}} 
 
\newcommand{\Z}{\mathbb{Z}}

\newcommand{\sO}{\mathcal{O}}

\theoremstyle{plain}
\newtheorem{thm}{Theorem}[section] 

\newtheorem{prop}[thm]{Proposition}

\newtheorem{lem}[thm]{Lemma}
\theoremstyle{definition} 
\newtheorem{defn}[thm]{Definition}

\newtheorem{obs}[thm]{Observation}
\theoremstyle{remark}
\newtheorem{rem}[thm]{Remark}
\newtheorem{note}[thm]{Note}
\newtheorem{ques}[thm]{Question}
\newtheorem{setting}[thm]{Setting}
\newtheorem{defn and notation}[thm]{Definition and Notation}
\newtheorem*{notation}{Notation} 
\newtheorem*{cl}{Claim}
\newtheorem*{clproof}{Proof of Claim}

\newtheorem*{acknowledgement}{Acknowledgments}

\begin{document}

\tolerance = 9999
\begin{abstract}
	When anti-canonical rings are finitely generated, we give a characterization of adjoint ideals using ultra-Frobenii, a characteristic zero analogue of Frobenius morphisms. This characterization enables us to give an alternative proof of a result of Zhuang, which states that if a ring is of klt type, then so is any of its pure subrings.
\end{abstract}
\maketitle
\section{Introduction}
 A ring homomorphism $R\to S$ is said to be {\it pure} if for any $R$-module $M$, the canonical morphism $M\to M\otimes_R S$ is injective. For example, if $R$ is a direct summand of $S$ as an $R$-module, then the inclusion $R\to S$ is pure. We say that a morphism $f:Y\to X$ of schemes is {\it pure} if for any $x\in X$, there exists $y\in Y$ such that $f(y)=x$ and the morphism $\sO_{X,x}\to \sO_{Y,y}$ is pure. Geometrically, when a linearly reductive group $G$ acts on a variety $Y$ over the field $\C$ of complex numbers, the quotient morphism $Y\to Y//G$ is a pure morphism. It is known that certain classes of singularities descend under pure morphisms. A classical result by Boutot \cite{Boutot87} states that rational singularities descend under pure morphisms. Schoutens \cite{Schoutens05} proved that if $Y\to X$ is a pure morphism between $\Q$-Gorenstein normal varieties over $\C$ and $Y$ has log terminal singularities, then so does $X$. However, $\Q$-Gorensteinness does not descend under pure morphisms even if $Y$ is nonsingular. Therefore, it is natural to generalize the definition of log terminal singularities to the non-$\Q$-Gorenstein setting, following the approach in \cite{dFH09}. Specifically, we say that a normal variety $X$ is of {\it klt type} if there exists an effective $\Q$-Weil divisor $\Delta$ on $X$ such that $K_X+\Delta$ is $\Q$-Cartier and the pair $(X,\Delta)$ has Kawamata log terminal (klt) singularities. Braun, Greb, Langlois, and Moraga \cite{BGLM24} showed that reductive quotients of singularities of klt type are of klt type. Zhuang \cite{Zhuang24} extended the results of Schoutens and Braun et al. and showed that if $Y\to X$ is pure and $Y$ is of klt type, then $X$ is of klt type. Takagi and the author \cite{TY24} generalized their result to the case of pairs $(X,\Delta)$, where $X$ is a normal variety and $\Delta$ is a $\Q$-Weil divisor. In order to show this theorem, they utilized reduction modulo $p>0$ and ultraproducts, a technique from non-standard analysis.
 
 $F$-singularities are singularities in positive characteristic, defined in terms of the Frobenius morphism. Strong $F$-regularity, one of the major classes of $F$-singularities, and test ideals, which play a central role in the theory of $F$-singularities, are positive characteristic analogues of klt singularities and multiplier ideals, respectively.
 
  The ultraproduct of a family of infinitely many sets is defined as a quotient of the product of these sets by an equivalence relation. For a local ring $R$ essentially of finite type over $\C$, Schoutens \cite{Schoutens03} constructed approximations $(R_p)_p$, families of local rings $R_p$ essentially of finite type over an algebraic closure $\overline{\mathbb{F}_p}$ of $\mathbb{F}_p$, and the non-standard hull $R_\infty$, the ultraproduct of $(R_p)_p$, such that there exists a natural inclusion $R\to R_\infty$. Within this framework, he \cite{Schoutens05} proved the aforementioned theorem about pure subrings of log terminal singularities by using ultra-Frobenii, which are induced by the product of iterated Frobenius morphisms $F^e:R_p\to R_p$. The author introduced in \cite{Yam23a} the notion of ultra-test ideals, which are defined similarly to test ideals but using ultra-Frobenii instead of usual Frobenius morphisms, and showed that they coincide with multiplier ideals if the ring is $\Q$-Gorenstein. In this paper, we generalize the definition of ultra-test ideals and also relax the assumption to include the case where the anti-canonical ring is finitely generated.
 	
 In singularity theory, the $\Q$-Gorenstein hypothesis often simplifies situations. A variety over an algebraically closed field of characteristic zero is said to be of {\it strongly $F$-regular type} if all its reductions modulo $p\gg 0$ are strongly $F$-regular. Hara and Watanabe \cite{Hara-Watanabe} showed that for pairs $(X,\Delta)$ such that $K_X+\Delta$ is $\Q$-Cartier, having klt singularities is equivalent to being of strongly $F$-regular type. Takagi \cite{Takagi} showed that reduction modulo $p\gg 0$ of the multiplier ideal $\mathcal{J}(X,\Delta)$ of a pair $(X,\Delta)$ coincides with the test ideal of its reductions $(X_p,\Delta_p)$ modulo $p\gg 0$ if $K_X+\Delta$ is $\Q$-Cartier. These two results were generalized by Chiecchio, Enescu, Miller and Schwede \cite{CEMS18} to the case where a anti-log-canonical ring is finitely generated. It is worth mentioning that this assumption is much more widely satisfied than the $\Q$-Gorenstein hypothesis. For example, such finite generation holds if the variety $X$ is of klt type (\cite{BCHM10}).
 
 The aim of this paper is to characterize adjoint ideals using ultra-Frobenii. Adjoint ideals are a generalization of multiplier ideals and define the non-purely-log-terminal loci. The main theorem is stated as follows:
 \begin{thm}[Theorem \ref{divisorial ultra-test ideals = adjoint ideals}]
 	Let $(R,\m)$ be a local normal domain essentially of finite type over $\C$, $D$ be a prime divisor or $D=0$,  $f$ be not in any minimal prime of $R(-D)$ and $t$ be a positive rational number. Suppose that $\bigoplus_{i\ge 0}R(-i(K_R+D))$ is a finitely generated $R$-algebra. Then
 	\[
 	\tau_{D}^{\mathrm{u}}(R,D,f^t)=\adj_D(R,D,f^t),
 	\]
 	where $\tau_D^{\mathrm{u}}(R,D)$ denotes the divisorial ultra-test ideal of the triple $(R,D,f^t)$ along $D$ (see Definition \ref{definition of divisorial ultra-test ideals}), and $\adj_D(R,D,f^t)$ denotes the adjoint ideal of the triple $(R,D,f^t)$ along $D$.
 \end{thm}
 The key idea behind the proof of the above theorem is to define divisorial ultra-test ideals, a variant of ultra-test ideals similar to divisorial test ideals introduced in \cite{Tak08}, and then to compare the top local cohomologies of the ring in question and its anti-canonical algebra as in \cite{Wat94}, where Watanabe proved that the anti-canonical algebra of a strongly $F$-regular ring is again strongly $F$-regular if the anti-canonical algebra is finitely generated. The divisorial ultra-test ideals differ from the concept of BCM test ideals used in \cite{TY24}, but they coincide when the ring is $\Q$-Gorenstein, because in that case, both are equal to adjoint ideals.
 As an application of the above theorem, we study the behavior of adjoint ideals for rings with finitely generated anti-canonical algebras under pure morphisms.
 \begin{thm}[Theorem \ref{adjoint ideals under pure morphisms}] \label{Main theorem 2}
 	Let $(R,\m)$ and $(S,\n)$ be normal local domains essentially of finite type over $\C$, let $D$ be a prime divisor or zero, let $\ba\subseteq R$ be a nonzero ideal such that, if $D$ is nonzero, its zero locus does not contain $D$ and let $t>0$ be a real number. Suppose that $\mathscr{R}=\bigoplus_{i\ge 0}R(-i(K_R+D))$ is a finitely generated $R$-algebra, $S$ is an $R$-algebra and $R\to S$ is a pure local $\C$-algebra homomorphism. Let $E$ be the Weil divisor on $\Spec S$ such that $S(-E)=(R(-D)\cdot S)^{**}$, where $(-)^{**}$ denotes the reflexive hull as an $S$-module, and suppose that $E$ is prime. Then we have
 	\[
 		\adj_E(S,E,(\ba S)^t)\cap R\subseteq \adj_D(R,D,\ba^t).
 	\]
 \end{thm}
This is a generalization of \cite[Theorem 1.2(2)]{TY24}, where one of the following conditions is assumed:
\begin{enumerate}
	\item $K_R+D$ is $\Q$-Cartier
	\item For any reduced divisor $B$ on $Y:=\Spec S$, $\bigoplus_{i\ge 0} S(iB)$ is a finitely generated $S$-algebra.
\end{enumerate}
 When $D=0$ and $\ba=R$, Theorem \ref{Main theorem 2} provides another proof of \cite[Theorem 1.1]{Zhuang24}, which states that any pure subring of singularities of klt type is again of klt type.
As an application of Theorem \ref{Main theorem 2}, we obtain a generalization of \cite[Corollary 1.3]{TY24}, where $X$ is assumed to be $\Q$-Gorenstein.
\begin{thm}[Theorem \ref{log canonical under pure morphisms}]
			Let $f:Y\to X$ be a pure morphism between normal quasi-projective varieties over $\C$ and suppose that $\bigoplus_{i\ge 0} \sO_X(-iK_X)$ is finitely generated.
	Assume in addition that one of the following conditions holds.
	\begin{enumerate}
		\item There exists an effective $\Q$-Weil divisor $\Delta$ on $Y$ such that $K_Y+\Delta$ is $\Q$-Cartier and no non-klt center of $(Y, \Delta)$ dominates $X$. 
		\item The non-klt-type locus of $Y$ has dimension at most one. 
	\end{enumerate}
	If $Y$ is of lc type, then $X$ is of lc type.
\end{thm}
Here a variety $X$ is said to be of {\it lc type} if there exists an effective $\Q$-Weil divisor $\Delta$ on $X$ such that $K_X+\Delta$ is $\Q$-Cartier and the pair $(X,\Delta)$ has log canonical singularities.

\begin{acknowledgement}
	The author would like to express his gratitude to Shunsuke Takagi for valuable discussion. The author was partially supported by JSPS KAKENHI Grant Number 24KJ1040.
\end{acknowledgement}

\section{Preliminaries}
\subsection{Adjoint ideals}
In this subsection, we quickly review the definition of multiplier ideals and adjoint ideal sheaves. We refer the reader to \cite{dFH09}, \cite{KM98}, \cite{LazarsfeldII}.

Let $X$ be a normal variety over an algebraically closed field $k$ of characteristic zero, $D$ be a reduced divisor on $X$, $\Delta$ be an effective $\Q$-Weil divisor on $X$ which has no common components with $D$, $\ba$ be a coherent ideal sheaf such that no components of $D$ is contained in the zero locus of $\ba$ and $t$ be a positive real number.

\begin{defn}
	\begin{enumerate}
		\item If $K_X+D+\Delta$ is $\Q$-Cartier, then take a log resolution $\pi:\tilde{X}\to X$ of $(X,D+\Delta,\ba)$ such that $\ba \sO_{\tilde{X}}=\sO_{\tilde{X}}(-F)$ for an effective Cartier divisor $F$ on $\tilde{X}$ and the strict transform $\pi_{*}^{-1}D$ of $D$ is smooth. Then the {\it adjoint ideal sheaf} $\adj_D(X,D+\Delta,\ba^t)$ of the triple $(X,D+\Delta,\ba^t)$ along $D$ is defined as
		\[
			\adj_D(X,D+\Delta,\ba^t)=\pi_*\sO_{\tilde{X}}(\lceil K_{\tilde{X}}-\pi^*(K_X+D+\Delta)-tF+\pi_*^{-1}D\rceil).
		\]
		The definition of the adjoint ideal is independent of the choice of log resolution. When $\ba=\sO_X$, we use $\adj_D(X,D+\Delta)$ to denote $\adj_D(X,D+\Delta,\ba^t)$.
		\item If $K_X+D+\Delta$ is not $\Q$-Gorenstein, then the {\it adjoint ideal} $\adj_D(X,D+\Delta,\ba^t)$ of the triple $(X,D+\Delta,\ba^t)$ is defined as
		\[
			\adj_D(X,D+\Delta,\ba^t)=\sum_{\Delta'}\adj_D(X,D+\Delta+\Delta',\ba^t),
		\]
		where $\Delta'$ runs through all effective $\Q$-Weil divisors on $X$ such that $D$ and $\Delta'$ have no common components and $K_X+D+\Delta+\Delta'$ is $\Q$-Cartier. When $D=0$, we use $\mathcal{J}(X,\Delta,\ba^t)$ to denote $\adj_D(X,D+\Delta,\ba^t)$ and call it the {\it multiplier ideal} of the triple $(X,\Delta,\ba^t)$.
		\item $X$ is said to be of {\it klt type} (resp. {\it lc type}) if there exists an effective $\Q$-Weil divisor $\Delta$ on $X$ such that $K_X+\Delta$ is $\Q$-Cartier and the pair $(X,\Delta)$ has klt singularities (resp. log canonical singularities).
	\end{enumerate}
\end{defn}
\begin{rem}
	$X$ is of klt type if and only if $\mathcal{J}(X)=\sO_X$, where $\mathcal{J}(X):=\mathcal{J}(X,0)$.
\end{rem}
\begin{prop}[{\cite[Remark 6.4]{dFH09},\cite[Proposition 2.6]{TY24}}]
	There exists an effective $\Q$-Weil divisor $\Delta'$ on $X$ such that $D$ and $\Delta'$ have no common components, $K_X+D+\Delta+\Delta'$ is $\Q$-Cartier and
	\[
		\adj_D(D,X+\Delta,\ba^t)=\adj_D(X,D+\Delta+\Delta',\ba^t).
	\]
\end{prop}

\subsection{Divisorial test ideals} In this subsection, we define divisorial test ideals. The reader is referred to \cite{Tak08} and \cite{Tak13} for details.
Let $R$ be a Noetherian normal domain of characteristic $p>0$ and $D$ be a reduced divisor on $X:=\Spec R$. Let $R^{\circ,D}$ denote the set of elements not in any minimal prime of $I_D:=R(-D)$. Let $\Delta$ be an effective $\Q$-Weil divisor on $X$ such that $D$ and $\Delta$ have no common components, $\ba$ be an ideal of $R$ such that $\ba \cap R^{\circ, D}\neq \emptyset$, and $t$ be a positive real number. We also suppose that $R$ is $F$-finite, i.e., the Frobenius morphism $F:R\to R$ is finite.
\begin{defn} \label{Frobenius pushforward}
	Let $M$ be an $R$-module and $e\ge 0$ be an integer. The {\it $e$-th iterated Frobenius pushforward} $F^e_*M$ of $M$ is defined as follows:
	\begin{enumerate}
		\item $F^e_*M$ is isomorphic to $M$ as an abelian group. We use $F^e_*x$ to denote the image of $x$ under this isomorphism.
		\item For $r\in R$ and $x\in M$, $r\cdot F^e_*x=F^e_*(r^{p^e}x)$.
	\end{enumerate}
\end{defn}
\begin{defn}
	The {\it divisorial test ideal} $\tau_D(R,D+\Delta,\ba^t)$ of the triple $(R,D+\Delta,\ba^t)$ along $D$ is defined as the smallest ideal $J$ of $R$ satisfying the following conditions:
	\begin{enumerate}
		\item $J\cap R^{\circ,D}\neq \emptyset$.
		\item For any integer $e\ge 0$ and any $\phi \in \Hom_R(F^e_*R(\lceil(p^e-1)(D+\Delta)\rceil),R) \subseteq \Hom_R(F^e_*R,R)$, one has $\phi(F^e_*(\ba^{\lceil t(p^e-1)\rceil}J))\subseteq J$.
	\end{enumerate}
	If $\ba=R$, we simply use $\tau_D(R,D+\Delta)$ to denote $\tau_D(R,D+\Delta,\ba^t)$.
\end{defn}
\begin{defn}
	Suppose that $(R,\m)$ is local and $\dim R=d$.
	For an $R$-module $M$, $0_{M}^{*_DD+\Delta,\ba^t}$ is an $R$-submodule of $M$ defined as follows:
	$z\in 0_{M}^{*_DD+\Delta,\ba^t}$ if and only if there exists an element $c\in R^{\circ,D}$ such that
	\[
		F^e_*(c\ba^{\lceil tp^e\rceil})\otimes z=0 \in F^e_*R((p^e-1)D+\lceil p^e\Delta\rceil)\otimes_R M.
	\]
\end{defn}
\begin{prop}[{cf. \cite[Proposition 8.23]{HH90}}] \label{divisorial test ideal characterization}
	Suppose that $R$ is local. The following ideals are equal to each other.
	\begin{enumerate}
		\item $\tau_D(R,D+\Delta,\ba^t)$.
		\item \[
			\bigcap_{M} \Ann_R 	0_{M}^{*_DD+\Delta,\ba^t},
		\]
		where $M$ runs through all $R$-modules.
		\item \[
		\Ann_R 0_{E}^{*_DD+\Delta,\ba^t},
		\]
		where $E=E_R(R/\m)$ is an injective hull of the residue field $R/\m$. 
	\end{enumerate}
\end{prop}
\begin{rem}
	The formation of divisorial test ideals commute with localization. Gluing together, we obtain divisorial test ideals for any $F$-finite normal integral schemes.
\end{rem}
 Let $X$ be a normal variety over $\C$, $D$ be an effective Weil divisor on $X$ and $\ba\subseteq \sO_X$ be a coherent ideal sheaf. $(A,X_A,D_A,\ba_A)$ is said to be a {\it model} of $(X,D,\ba)$ if the following conditions hold:
 \begin{enumerate}
 	\item $A$ is a finitely generated $\Z$-subalgebra of $\C$.
 	\item $X_A$ is a scheme of finite type over $A$ such that $X_A\times_{\Spec A}\Spec \C\cong X$.
 	\item $D_A$ is a closed subscheme of $X_A$ such that $D_A\times_{\Spec A}\C \cong D$.
 	\item $\ba_A\subseteq \sO_{X_A}$ is a coherent ideal sheaf such that $\ba_A\sO_X=\ba$.
 \end{enumerate}
 Enlarging $A$, we may assume that $X_A$ is normal and $D_A$ is a Weil divisor on $X_A$. For any closed point $\mu\in \Spec A$, let $X_\mu:=X_A\times_{\Spec A}\Spec A/\mu$, $D_\mu:=D_A\times_{\Spec A}\Spec A/\mu$ and $\ba_\mu:=\ba_A\sO_{X_\mu}$. $X_\mu$, $D_\mu$ and $\ba_\mu$ are said to be a {\it reduction modulo $p>0$} of $X$, $D$ and $\ba$ respectively. A reduction modulo $p>0$ of $\Q$-Weil divisors are defined similarly. In the same framework, we can reduce finitely many tuples of varieties, $\Q$-Weil divisors and ideals to characteristic $p>0$. 
 
 Adjoint ideals are corresponding to divisorial test ideals under reduction modulo $p>0$.
 \begin{prop}[{\cite[Theorem 7.3]{TY24}}]
 	Suppose that $X$ is a normal variety over $\C$, $D$ is a reduced divisor on $X$, $\Delta$ is an effective $\Q$-Weil divisor on $X$ such that $D$ and $\Delta$ have no common components and $\bigoplus_{i\ge0}\sO_X(\lfloor -i(K_X+D+\Delta)\rfloor)$ is finitely generated. Let $\ba\subseteq \sO_X$ be a coherent ideal sheaf whose zero locus contains no components of $D$ and $t>0$ be a real number. Then
 	\[
 		\adj_D(X,D+\Delta,\ba^t)_{\mu}=\tau_{D_\mu}(X_\mu,D_\mu+\Delta_\mu,\ba_\mu^t)
 	\]
 	for general closed point $\mu\in \Spec A$.
 \end{prop}
\subsection{Ultraproducts} In this subsection, we recall the definition of ultraproducts and their basic properties. We refer the reader to \cite{Schoutens03} and \cite[Section 3]{Yam23} for details.

Let $\mathcal{P}$ be the set of prime numbers.
\begin{defn}
	A non-empty subset $\mathcal{F}$ of the power set of $\mathcal{P}$ is said to be a {\it non-principal ultrafilter} on $\mathcal{P}$ if the following four conditions hold:
	\begin{enumerate}
		\item If $A, B\in \mathcal{F}$, then $A\cap B\in \mathcal{F}$.
		\item If $A\in \F$ and $A\subseteq B \subseteq \mathcal{P}$, then $B\in \mathcal{F}$.
		\item If $A$ is a finite subset of $\mathcal{P}$, then $A\notin \mathcal{P}$.
		\item If $A$ is a subset of $\mathcal{P}$, then we have $A\in \mathcal{F}$ or $\mathcal{P}\setminus A\in \mathcal{F}$.
	\end{enumerate}
\end{defn}
\begin{defn}
	Fix a non-principal ultrafilter on $\mathcal{P}$. For any property $\phi$ on $\mathcal{P}$, we say that $\phi(p)$ holds {\it for almost all $p$} if $\{p\in \mathcal{P}|\text{$\phi(p)$ holds}\}\in \mathcal{F}$. This depends on the choice of $\mathcal{F}$.
\end{defn}
\begin{defn}
	Fix a non-principal ultrafilter $\mathcal{F}$ on $\mathcal{P}$.
	Let $(A_p)_p$ be a family of non-empty sets indexed by $\mathcal{P}$. The {\it ultraproduct} $A_\infty$ of $(A_p)_p$ is defined as
	\[
		A_\infty=\ulim_p A_p:=\left(\prod_p A_p\right)/\sim,
	\]
	where $\sim$ is the equivalent relation on $\prod_p A_p$ such that
	\[
		(a_p)\sim (b_p)
	\]
	if and only if $\{p\in\mathcal{P}|a_p=b_p\}\in \mathcal{F}$. The equivalent class of a sequence $(a_p)\in \prod_p A_p$ is denoted by $\ulim_p a_p$.
\end{defn}
\begin{rem}
	If each $A_p$ is a ring, then $A_\infty$ is a quotient ring of $\prod_p A_p$. If each $M_p$ is an $A_p$-module, then $M_\infty$ is an $A_\infty$-module. See \L o\'{s}'s theorem for details.
\end{rem}
\begin{notation}
	We use ${^*}\N$ (resp. ${^*}\Z, {^*}\Q$) to denote $\ulim_p \N$ (resp. $\ulim_p \Z$, $\ulim_p\Q$) and use $\pi$ to denote $\ulim_p p\in {^*}\N$. Note that ${^*}\N$ is an ordered semiring and ${^*}\Z$, ${^*}\Q$ are ordered rings.
\end{notation}
	The following theorem is a key observation to apply ultraproducts to commutative algebra.
\begin{thm}
	Let $\overline{\mathbb{F}_p}$ be an algebraic closure of the finite field $\mathbb{F}_p$ for $p\in \mathcal{P}$. Then we have a field isomorphism
	\[
		\ulim_p \overline{\mathbb{F}_p}\cong \C.
	\]
\end{thm}
Using this isomrphism, for any $n\in \N$, we can construct a $\C$-algebra homomorphism
\[
	\C[X_1,\dots,X_n]\to \ulim_p (\overline{\mathbb{F}_p}[X_1,\dots,X_n]),
\]
where $X_i$ is mapped to $X_i:=\ulim_p (X_i)_p$. We use $\C[X_1,\dots,X_n]_\infty$ to denote $\ulim_p (\overline{\mathbb{F}_p}[X_1,\dots,X_n])$.
\begin{thm}[\cite{vdD79}]
	The above morphism is faithfully flat.
\end{thm}
	Moreover, the following statement holds.
\begin{prop}[\cite{vdD79}]
	If $\mathfrak{p}$ is a prime ideal of $\C[X_1,\dots,X_n]$, then $\mathfrak{p}\C[X_1,\dots,X_n]_\infty$ is also prime.
\end{prop}
\begin{defn}
	\begin{enumerate}
		\item For any element $f$ of $\C[X_1,\dots,X_n]$, a sequence $(f_p)$ of elements of $\overline{\mathbb{F}_p}[X_1 ,\dots,X_n]$ is said to be an {\it approximation} of $f$ if $f=\ulim_p f_p$ in $\C[X_1,\dots,X_n]_\infty$.
		\item For an ideal $I=(f_1,\dots,f_m)$ of $\C[X_1,\dots,X_n]$, a sequence $(I_p)$ of ideals $I_p=(f_{1,p},\dots, f_{m,p})$ is said to be an {\it approximation} of $I$ if $(f_{i,p})$ is an approximation of $f_i$ for any $i$. 
	\end{enumerate}
\end{defn}
	Based on these observations, Schoutens introduced the notion of approximations and non-standard hulls.
\begin{defn}
	Let $R$ be a local ring essentially of finite type over $\C$. Take $n\in \N$, an ideal $I\subseteq \C[X_1,\dots,X_n]$ and a prime ideal $\mathfrak{p}$ of $\C[X_1,\dots,X_n]$ containing $I$ such that $R\cong (\C[X_1,\dots,X_n]/I)_{\mathfrak{p}}$. The {\it non-standard hull} $R_\infty$ of $R$ is defined as
	\[
		R_\infty=(\C[X_1,\dots,X_n]_\infty/I\C[X_1,\dots,X_n]_\infty)_{\mathfrak{p}\C[X_1,\dots,X_n]_\infty}.
	\]
	The isomorphism class of non-standard hulls of $R$ is independent of the choice of presentation of $R$. If $(\mathfrak{p}_p)$ and $(I_p)$ are approximations of $\mathfrak{p}$ and $I$ respectively, then an {\it approximation} $(R_p)$ of $R$ is a family of local rings such that
	\[
		R_p=(\overline{\mathbb{F}_p}[X_1,\dots,X_n]/I_p)_{\mathfrak{p}_p}
	\]
	for almost all $p$. Note that $\mathfrak{p}_p$ is prime for almost all $p$ and two approximations are isomorphic for almost all $p$.
\end{defn}
\begin{defn} Let $R$ be a local ring essentially of finite type over $\C$.
	\begin{enumerate}
		\item Let $f$ be an element of $R$. A family of elements $(f_p)$ of $R_p$ is said to be an {\it approximation} of $f$ if $f=\ulim_p f_p$ in $R_\infty$.
		\item For an ideal $I=(f_1,\dots,f_m)$ of $R$, a sequence $(I_p)$ of ideals $I_p=(f_{1,p},\dots,f_{m,p})$ is said to be an {\it approximation} of $I$ if $(f_{i,p})$ is an approximation of $(f_i)$ for any $i$.
	\end{enumerate}
\end{defn}
In this setting, we can construct morphisms analogous to Frobenius.
\begin{defn}
	\begin{enumerate}
		\item For $\epsilon\in {^*\N}$, we define the {\it ultra-Frobenius} $F^\epsilon$ corresponding to $\epsilon$ as
		\[
		F^\epsilon:R_\infty \to R_\infty; x\to \ulim_p x_p^{p^{e_p}},
		\]
		where $(x_p)$ denotes an approximation of $x$.
		We also use $F^\epsilon$ to denote the composite morphism $R\hookrightarrow R_\infty\xrightarrow{F^\epsilon} R_\infty$.
		\item For an $R_\infty$-module $M$, we use $F^\epsilon_* M$ to denote the restriction of scalars along $F^\epsilon$ (cf. Definition \ref{Frobenius pushforward}).
	\end{enumerate}
\end{defn}
\begin{defn}
	Let $R$ be a local ring essentially of finite type over $\C$ and $M$ be a finitely generated $R$-module. Suppose that $M$ fits into an exact sequence
	\[
		R^b\xrightarrow{A} R^a \to M \to 0,
	\]
	where $a,b\ge 1$ and $A$ is an $a\times b$ matrix whose entries are in $R$. Let $(R_p)$ be an approximation of $R$. Then a family $(M_p)$ of $R_p$-modules is said to be an {\it approximation} of $M$ if $M_p$ is the cokernel of a morphism
	\[
		R_p^{b}\xrightarrow{A_p} R_p^a
	\]
	for almost all $p$, where $(A_p)$ is a family of matrices defined by an entrywise approximation.
	Approximations $(M_p)$ are independent of the choice of the matrix $A$ up to almost equality and isomorphisms.
\end{defn}
\begin{rem}
	If $(M_p)$ is an approximation of $M$, then we have an isomorphism
	\[
		M\otimes_R R_\infty \cong \ulim_p M_p
	\]
	by construction.
\end{rem}
We can also define approximations of Weil divisors on $R$ using approximations of ideals or approximations of modules. Moreover, we can generalize the notion of divisors as follows.
\begin{defn}
	\begin{enumerate}
		\item We say that $\Delta$ is {\it Weil ultra-divisor} if $\Delta$ is a formal sum of prime divisors with ${^*}\Z$-coefficients. Similarly, $\Delta$ is called {\it ${^*}\Q$-Weil ultra-divisor} if $\Delta$ is a formal sum of prime divisors with ${^*}\Q$-coefficients.
		\item Let $D_i$ be distinct prime divisors, $a_i=\ulim_p a_{i,p}\in {^*}\Q$ and $\Delta=\sum_{i} a_i D_i$. Let $\lfloor \Delta \rfloor$ denote $\sum_i \lfloor a_i\rfloor D_i$, where $\lfloor a_i\rfloor$ is defined as $\ulim_p \lfloor a_{i,p} \rfloor$. $\lceil \Delta \rceil$ is defined similarly.
		\item Let $D_i$ be prime divisors, $a_i=\ulim_p a_{i,p}\in {^*}\Z$ and $\Delta=\sum_{i} a_i D_i$. Then $(\Delta_p)$ is an approximation of $\Delta$ if $(D_{i,p})$ is an approximation of $D_{i}$ for any $i$, $(a_{i,p})$ is an approximation of $a_i$ for any $i$, and $\Delta_p=\sum_i a_{i,p}\Delta_p$ for almost all $p$.
		\item For a Weil ultra-divisor $\Delta$, let $R_\infty(\Delta)$ denotes $\ulim_p R_p(\Delta_p)$, where $(\Delta_p)$ is an approximation of $\Delta$.
	\end{enumerate}
\end{defn}
\begin{obs}
	Let $R$ be a local ring essentially of finite type over $\C$. Here we consider local cohomologies of approximations (see \cite[Section 5]{Scho08}).
	Let $x_1,\dots,x_d$ be a system of parameters for $R$. Suppose that $M_p$ is an $R_p$-module for almost all $p$ and $M_\infty=\ulim_p M_p$. For an integer $n \ge 0$ and an $n$-tuple $1\le i_1<\dots<i_n\le d$, there exists a natural map
	\[
	(M_\infty)_{x_{i_1}\cdots x_{i_d}}\to \ulim_p (M_p)_{x_{i_1,p}\dots x_{i_d,p}}.
	\]
	Considering the \v{C}ech complexes of $M_p$ and $M_{\infty}$, we have commutative diagrams
	\[
	\xymatrix{
		\bigoplus_{1\le i_1<\cdots<i_n \le d}(M_\infty)_{x_{i_1}\dots x_{i_n}} \ar[r] \ar[d]& \bigoplus_{1\le j_1<\cdots<j_{n+1}\le d} (M_\infty)_{x_{j_1}\dots x_{j_{n+1}}} \ar[d]\\
		\bigoplus_{1\le i_1<\cdots<i_n \le d}\ulim_p (M_p)_{x_{i_1,p}\dots x_{i_n,p}} \ar[r] & \bigoplus_{1\le j_1<\cdots<j_{n+1}\le d}\ulim_p (M_p)_{x_{j_1,p}\dots x_{j_{n+1},p}}, 
	}
	\]
	which yield natural maps 
	\[
	H_{\m}^i(M_\infty)\to \ulim_p H_{\m_p}^i(M_p).
	\]
\end{obs}
	We can extend the above definitions to a relative setting.
	Let $R$ be a local ring essentially of finite type over $\C$.
	For an approximation $(R_p)$ of $R$ and a positive integer $n$, we have a faithfully flat map
	\[
		R[X_1,\dots,X_n] \to \ulim_p \left(R_p[X_1,\dots,X_n]\right).
	\]
	We can define an approximation of elements of $R[X_1,\dots,X_n]$ and ideals of $R[X_1,\dots, X_n]$.
	Let $S$ be a finitely generated $R$-algebra. Suppose that $S\cong R[X_1,\dots,X_n]/I$.
	A {\it relative approximation} $S_p$ of $S$ is defined to be $R_p[X_1,\dots,X_n]/I_p$ and the {\it relative hull} $S_\infty$ of $S$ is defined to be $\ulim_p S_p$.
	If $S$ is an $\N$-graded algebra, $S_p$ is $\N$-graded for almost all $p$. For an element $f\in S_\infty$ and $\nu \in {^*\N}$, we define the {\it degree $\nu$ part} $f_\nu$ of $f$ as
	\[
		f_\nu:=\ulim_p (f_p)_{n_p},
	\]
	where $f=\ulim_p f_p$, $\nu=\ulim_p n_p$ and $(f_p)_{n_p}$ is the degree $n_p$ part of $f_p$.
\section{Main theorem}
\begin{setting} \label{setting of rings with finitely generated anti-canonical algebras}
	Let $(R,\m)$ be a local normal domain essentially of finite type over $\C$ of dimension $d$, $D$ be a prime divisor or $D=0$,  $f$ be an element of $R^{\circ,D}$ and $t$ be a positive rational number. Fix a canonical divisor $K_R$ of $\Spec R$ such that $K_R+D$ is effective and has no component equal to $D$. Suppose that $\mathscr{R}=\bigoplus_{i\ge 0}R(-i(K_R+D))T^i\subseteq R[T]$ is a finitely generated $R$-algebra. Let $\mathscr{M}=\m\mathscr{R}+\mathscr{R}_{+}$ be the unique homogeneous maximal ideal of $\mathscr{R}$. 
\end{setting}
\begin{defn} \label{definition of divisorial ultra-test ideals}
	With notation as in Setting \ref{setting of rings with finitely generated anti-canonical algebras}, let $M$ be an $R$-module. $0_{M}^{*\mathrm{u}_DD,f^t}$ is defined as follows: For $\eta\in M$, $\eta\in 0_{M}^{*\mathrm{u}_DD,f^t}$ if and only if there exists $c\in R^{\circ,D}$ such that for any $\epsilon \in {^*\N}$, $\eta\otimes F_*^{\epsilon}(cf^{\lceil t\pi^\epsilon \rceil})=0$ in $M\otimes_R F^\epsilon_*R_{\infty}((\pi^\epsilon-1)D)$. The {\it divisorial ultra-test ideal} $\tau_D^{\mathrm{u}}(R,D,f^t)$ of the triple $(R,D,f^t)$ is defined to be $\Ann_R 0_{E}^{*\mathrm{u}_DD,f^t}$, where $E$ is the injective hull of the residue  field of $R$.
\end{defn}
\begin{rem}
	If $D=0$, the divisorial ultra-test ideal $\tau_{D}^{\mathrm{u}}(R,D,f^t)$ of a triple $(R,D,f^t)$ coincides with the ultra-test ideal $\tau_{\mathrm{u}}(R,f^t)$ of the pair $(R,f^t)$ defined in \cite{Yam23a}.
\end{rem}
\begin{prop}\label{tau over all modules}
	With notation as in Setting \ref{setting of rings with finitely generated anti-canonical algebras}, we have
	\[
		\tau_{D}^{\mathrm{u}}(R,D,f^t)=\bigcap_{M} \Ann_{R} 0_{M}^{*\mathrm{u}_DD,f^t},
	\]
	where $M$ runs through all $R$-modules.
\end{prop}
\begin{proof}
	This follows from an argument similar to \cite[Proposition 5.11]{Yam23a}, which is essentially the same proof as \cite[Proposition 8.23]{HH90}.
\end{proof}
\begin{prop} \label{adj subset tau}
	With notation as in Setting \ref{setting of rings with finitely generated anti-canonical algebras}, we have
	\[
		\adj_D(R,D,f^t)\subseteq \tau^{\mathrm{u}}(R,D,f^t).
	\]
\end{prop}
\begin{proof}
	Let $\eta \in 0_{H_\m^d(\omega_R)}^{*\mathrm{u}_DD,f^t}$ and $a\in \adj_D(R,D,f^t)$. By \cite[Theorem 7.3]{TY24} and an argument similar to \cite[Proposition 5.5]{Yam23}, we have $(\adj_D(R,D,f^t))_p=\tau_{D_p}(R_p,D_p,f_p^t)$ for almost all $p$. Hence, $a_p\in \tau_{D_p}(R_p,D_p,f_p^t)$ for almost all $p$. Suppose that $\eta_p \notin 0_{H_{\m_p}^d(\omega_{R_p})}^{*_{D_p}D_p,f_p^t}$ for almost all $p$. Take $c\in R^{\circ,D}$ such that $\eta\otimes F^\epsilon_*(cf^{\lceil t\pi^\epsilon \rceil})=0$ in $H_\m^d(\omega_R)\otimes_R F^\epsilon_*R_{\infty}((\pi^\epsilon-1)D)$ for any $\epsilon\in {^*}\N$. For almost all $p$, there exists $e_p\in \N$ such that $\eta_p \otimes F^{e_p}_*(c_pf_p^{\lceil tp^{e_p}\rceil})\neq 0$ in $H_{\m_p}^d(\omega_{R_p})\otimes_{R_p}F^{e_p}_*R_p((p^{e_p}-1)D_p)$. Let $\epsilon:=\ulim_p e_p$. Since we have a natural morphism 
	\[
	H_\m^d(\omega_R)\otimes_R F^\epsilon_*R_{\infty}((\pi^\epsilon-1)D) \to \ulim_p(H_{\m_p}^d(\omega_{R_p})\otimes_{R_p}F^{e_p}_*R_p((p^{e_p}-1)D_p)),
	\] and $\eta\otimes F^\epsilon_*(cf^{\lceil t\pi^\epsilon \rceil})=0$ is mapped to $\ulim_p (\eta_p\otimes F^{e_p}_*(c_pf_p^{\lceil tp^{e_p} \rceil}))\neq 0$, which is a contradiction. Hence, $\eta_p\in 0_{H_{\m_p}^d(\omega_{R_p})}^{*_{D_p}D_p,f_p^t}$ for almost all $p$. Therefore, $a_p\eta_p=0$ for almost all $p$. Since $H_\m^d(\omega_R)\to \ulim_p H_{\m_p}^d(\omega_{R_p})$ is injective (see \cite[Proposition 3.3, Proposition 3.5, Proof of Proposition 3.9]{Yam24}), we have $a\eta=0$.
\end{proof}
\begin{lem}
		With notation as in Setting \ref{setting of rings with finitely generated anti-canonical algebras}, let $\mathscr{D}$ be the Weil divisor on $\Spec \mathscr{R}$ defined as $\mathscr{R}(-\mathscr{D})=(R(-D)\cdot\mathscr{R})^{**}$, where $(-)^{**}$ is the reflexive hull as an $\mathscr{R}$-module. Then $\mathscr{D}$ is prime and $\mathscr{R}(-\mathscr{D})=\bigoplus_{i\ge 0}R(-D-i(K_R+D))T^{i}$.
\end{lem}
\begin{proof}
	If $D=0$, this is clear. Suppose that $D$ is prime. By \cite[Lemma 4.3 (4)]{GHNV90}, there exists a unique height one prime of $\mathscr{R}$ containing $R(-D)$. Hence, $\mathscr{D}$ is prime. Let $\mathfrak{p}=\bigoplus_{i\ge 0}R(-D-i(K_R+D))T^{i}$. Since $\mathfrak{p}=(R(-D)\cdot R[T])\cap \mathscr{R}$, $\mathfrak{p}$ is prime. Since $\mathscr{R}(-\mathscr{D})$ is a unique minimal prime of $\mathscr{R}$ containing $R(-D)$, we get $\mathscr{R}(-\mathscr{D})\subseteq \mathfrak{p}$. It is enough to show that $\mathscr{R}(-\mathscr{D})$ and $\mathfrak{p}$ is isomorphic in codimension one. Let $U$ be the regular locus of $\Spec R$ and $\pi:\Spec \mathscr{R}\to \Spec R$. Then the complement of $\pi^{-1}U$ is of codimension $\ge 2$. Indeed, if $P\in \Spec \mathscr{R}\setminus \pi^{-1}U$, then $\pi(P)\in \Spec R\setminus U$ is corresponding to a prime ideal of height $\ge 2$, but this contradicts \cite[Lemma 4.3 (2)]{GHNV90}. We have 
	\[
		\mathscr{R}(-\mathscr{D})|_{\pi^{-1}U}=(R(-D)\mathscr{R})|_{\pi^{-1}U}=\mathfrak{p}|_{\pi^{-1}U}.
	\]
\end{proof}
\begin{obs}
 By \cite[Theorem 4.5]{GHNV90}, $K_{\mathscr{R}}+\mathscr{D}$ is linearly equivalent to zero. Hence, up to some degree shift, we have
	\begin{align*}
		H_{\mathscr{M}}^{d+1}(\omega_{\mathscr{R}}) &\cong H_{\mathscr{M}}^{d+1}(\mathscr{R})\otimes_R R(-D)\\
		&\cong (\bigoplus_{i>0}H_\m^{d}(R(i(K_R+D)))T^{-i})\otimes_R R(-D) \\
		& \cong \bigoplus_{i>0}H_\m^{d}(R(-D+i(K_R+D)))T^{-i},
	\end{align*}
	where the second isomorphism follows from \cite[Theorem 2.2]{Wat94}, and the third isomorphism follows from the fact that $R(i(K_R+D))\otimes_{R}R(-D)$ and $R(-D+i(K_R+D))$ are isomorphic in codimension one. We regard $\bigoplus_{i\ge 0}R(i(K_R+D))T^{-i}$ as an $\mathscr{R}$-module, where any component of positive degree is considered to be zero. Then we have an $\mathscr{R}$-module isomorphism
	\[
		H_\m^{d}(\omega_R)\otimes_R \left(\bigoplus_{i\ge 0}R(i(K_R+D))T^{-i}\right)\xrightarrow[\cong]{\cdot T^{-1}} \bigoplus_{i>0}H_\m^{d}(R(-D+i(K_R+D)))T^{-i}.
	\]
\end{obs}
\begin{prop}
	With notation as above, for any $\epsilon\in {^*\Z}$ and any $\alpha\in {^*\Z}$, we have an $R$-linear homomorphism
	\[
		\psi_{\epsilon,\alpha}:\left(\bigoplus_{i\ge 0}R(i(K_R+D))T^{-i}\right)\otimes_{\mathscr{R}}F^{\epsilon}_* (\mathscr{R}_{\mathscr{M}})_\infty((\pi^\epsilon-1)\mathscr{D}) \to F^\epsilon_*R_\infty((\pi^\epsilon-1)D)
	\]
	which satisfies the following:
	\[
	aT^{-i}\otimes F^{\epsilon}_{*}\xi \mapsto \left\{ 
		\begin{array}{ll}
			0 & (\alpha\ge \pi^\epsilon) \\
			\ulim_p F^{e_p}_* (a_p^{p^{e_p}}[\xi_p]_{\alpha_p+p^{e_p}i}) & (\text{otherwise})
		\end{array}
		\right.,
	\]
	where $a=\ulim_p a_p$, $\alpha=\ulim_p \alpha_p$ and $[\xi_p]_{\alpha_p+p^{e_p}i}$ denotes the $(\alpha_p+p^{e_p}i)$-th homogeneous part of $\xi_p$. 
\end{prop}
\begin{note}
	If $R$ is an $\N$-graded ring and $M$ is its unique maximal homogeneous ideal, we have the following commutative diagram:
	\[
		\xymatrix{
			R \ar[r] \ar[rd]& R_{M} \ar[d]\\
			& \widehat{R}^{R_{+}} 
		}.	
	\]
	Here $\widehat{R}^{R_{+}}$ is isomorphic to $\prod_{i\ge 0} R_i$ as an $R$-module. For an element $f\in R_{M}$ and $i\in \N$, {\it the $i$-th homogeneous part of $f$} is defined to be the image of $f$ under the morphism $R_{M}\to \widehat{R}^{R_+}\to R_i$. If $i<0$,  the $i$-th homogeneous part of $f$ is defined to be zero.
\end{note}
\begin{proof}
	Suppose that $\alpha<\pi^\epsilon$.
	For any $i\in \N$, $a\in R(i(K_R+D))$ and $\xi \in (\mathscr{R}_{\mathscr{M}})_\infty((\pi^\epsilon-1)\mathscr{D})$, we have
	$a_p^{p^{e_p}}[\xi_p]_{\alpha_p+p^{e_p}i}\in R_p((p^{e_p}-1)D_p)$ for almost all $p$. Indeed, 
	\begin{align*}
	p^{e_p}i(K_R+D)+(p^e-1)D_p-(\alpha_p+p^{e_p}i)(K_{R_p}+D_p) &= (p^{e_p}-1)D_p-\alpha_p(K_{R_p}+D_p) \\
	&\le (p^{e_p}-1)D_p
	\end{align*}
	for almost all $p$.
	Hence, we have an $R$-linear homomorphism
	$R(i(K_R+D))\otimes_{R}F^{\epsilon}_* (\mathscr{R}_{\mathscr{M}})_\infty((\pi^\epsilon-1)\mathscr{D}) \to F^\epsilon_*R_\infty((\pi^\epsilon-1)D)$.
	Summing up these morphisms, we get an $R$-linear homomorphism
	\[
	\left(\bigoplus_{i\ge 0}R(i(K_R+D))T^{-i}\right)\otimes_{R}F^{\epsilon}_* (\mathscr{R}_{\mathscr{M}})_\infty((\pi^\epsilon-1)\mathscr{D}) \to F^\epsilon_*R_\infty((\pi^\epsilon-1)D).
	\]
	It is enough to show that the above morphism factors through the tensor over $\mathscr{R}$.
	For any $i,j\in \N$, $a\in R(i(K_R+D))$, $b\in R(-j(K_R+D))$ and $\xi \in (\mathscr{R}_{\mathscr{M}})_\infty((\pi^\epsilon-1)\mathscr{D})$, we need to show that $abT^{j-i}\otimes F^{\epsilon}_*\xi$ and $aT^{-i}\otimes (bT^{j}\cdot F^\epsilon_*\xi)$ have the same image. Indeed, we have
	\[
		(a_pb_p)^{p^{e_p}}[\xi_p]_{\alpha_p+p^{e_p}(i-j)} = a_p^{p^{e_p}} [b_p^{p^{e_p}}T^{jp^{e_p}}\xi_p]_{\alpha_p+p^{e_p}i}.
	\]
\end{proof}
\begin{prop} \label{inifinite cyclic covers tau}
	With notation as in \ref{setting of rings with finitely generated anti-canonical algebras}, we have
	\[
		\tau_D^{\mathrm{u}}(R,D,f^t)\subseteq \tau_{\mathscr{D}}^{\mathrm{u}}(\mathscr{R}_{\mathscr{M}},\mathscr{D},f^t).
	\]
\end{prop}
\begin{proof}
	Take $\eta\in 0_{H_{\mathscr{M}}^{d+1}(\omega_{\mathscr{R_{\mathscr{M}}}})}^{*\mathrm{u}_{\mathscr{D}}\mathscr{D},f^t}$ and $x\in \tau_{D}^{\mathrm{u}}(R,D,f^t)$. Suppose that $x\eta\neq 0$. We may assume that $x\eta\in \operatorname{Soc}_{\mathscr{R}}H^{d+1}_{\mathscr{M}}(\omega_{\mathscr{R}})$. Then $x\eta$ is a homogeneous element. There exists $c\in \mathscr{R}^{0,\mathscr{D}}$ such that for any $\epsilon \in {^*\N}$, we have $\eta\otimes F^{\epsilon}_*(cf^{\lceil t\pi^\epsilon \rceil})=0$ in $H_{\mathscr{M}}^{d+1}(\omega_{\mathscr{R}})\otimes_{\mathscr{R}}F^{\epsilon}_*(\mathscr{R}_{\mathscr{M}})_\infty ((\pi^\epsilon-1)\mathscr{D})$. Let $i\in \N$ be an integer such that the $i$-th homogeneous part $c_i$ of $c$ is not in $R(-D-i(K_R+D))$ if $D$ is prime and is nonzeo if $D=0$, and let $\eta_{-1}$ be the $({-1})$-th homogeneous part of the image of $\eta$ under the isomorphism $H_{\mathscr{M}}^{d+1}(\omega_{\mathscr{R}})\to \bigoplus_{i>0}H_{\m}^d(R(-D+i(K_R+D)))T^{-i}$. Then $c_i\in R^{\circ,D}$ because $K_R+D$ has no component equal to $D$, $c_i\in R(-i(K_R+D))$ and $c_i\notin R(-D-i(K_R+D))$. Considering the following composite morphism
	\begin{align*}
		H_{\mathscr{M}}^{d+1}(\omega_{\mathscr{R}}) &\to H_{\mathscr{M}}^{d+1}(\omega_{\mathscr{R}})\otimes_{\mathscr{R}}F^{\epsilon}_*(\mathscr{R}_{\mathscr{M}})_\infty ((\pi^\epsilon-1)\mathscr{D}) \\
		& \xrightarrow{\cong} H_\m^d(\omega_R)\otimes_R\left(\bigoplus_{i\ge 0}R(i(K_R+D))T^{-i}\right)\otimes_{\mathscr{R}}F^{\epsilon}_*(\mathscr{R}_{\mathscr{M}})_\infty ((\pi^\epsilon-1)\mathscr{D}) \\
		& \xrightarrow{\operatorname{id}\otimes\psi_{\epsilon,i}} H_\m^d(\omega_R)\otimes_R F^\epsilon_*R_\infty((\pi^\epsilon-1)D),
	\end{align*}
	we have $\eta \mapsto \eta\otimes F^{\epsilon}_*(cf^{\lceil t\pi^\epsilon \rceil})=0 \mapsto \eta_{-1}\otimes F^{\epsilon}_*(c_if^{\lceil t\pi^\epsilon \rceil})$. Hence, $\eta_{-1}\in 0_{H_{\m}^d(\omega_R)}^{*\mathrm{u}_DD,f^t}$. Hence, $x\eta_{-1}=0$. Since $x\eta \in  \operatorname{Soc}_{\mathscr{R}}H^{d+1}_{\mathscr{M}}(\omega_{\mathscr{R}})$ is a nonzero homogeneous element, $x\eta$ maps to $x\eta_{-1}$ under the isomorphism $H_{\mathscr{M}}^{d+1}(\omega_{\mathscr{R}})\to \bigoplus_{i>0}H_{\m}^d(R(-D+i(K_R+D)))T^{-i}$, which is a contradiction.
\end{proof}
\begin{prop} \label{Q-Gorenstein case}
	Let $R$ be a local normal domain essentially of finite type over $\C$, $D$ be a prime divisor on $\Spec R$ or $D=0$, $f\in R^{\circ,D}$ and $t$ be a positive rational number. Suppose that $K_R+D$ is Cartier. Then we have
	\[
		\tau_{D}^{\mathrm{u}}(R,D,f^t)=\adj_D(R,D,f^t).
	\]
\end{prop}
\begin{proof}
	Let $\eta\in \Ann_{H_{\m}^d(\omega_R)}\adj_{D}(R,D,f^t)$. By \cite[Theorem 5.7]{TY24}, $\eta \in 0_{H_\m^d(\omega_R)}^{\mathcal{B}(I_D,D+t\Div (f))}$ (see \cite[Definition 5.2, Definition 5.4]{TY24}). By the following commutative diagram
	\[
		\xymatrix{
			H_{\m}^d(\omega_R) \ar[r] \ar[d]& H_\m^d(\omega_R)\otimes_R \mathcal{B}(I_D,D+t\Div(f)) \ar[d]\\
			\ulim_{p} H_{\m_p}^d(\omega_{R_p}) \ar[r] & \ulim_p H_{\m_p}^d(\omega_{R_p})\otimes_{R_p}I_{D_p}^+(D_p+t \Div (f_p))
		}
	\]
	(see \cite[Definition 4.2]{TY24} for the definition of $I_{D_p}^+(D_p+t\Div(f_p))$),
	we have 
	\[
		\eta_p\in 0_{H_{\m_p}^d(\omega_{R_p})}^{+_{D_p}(D_p+t\Div(f_p))}=0_{H_{\m_p}^d(\omega_{R_p})}^{*_{D_p}(D_p+t\Div(f_p))}
	\]
	for almost all $p$. Let $g\in R^{\circ, D}$ such that $\Div(g)=K_R+D$. We have an $R$-linear homomorphism
	\begin{align*}
		H_{\m}^d(R)\otimes_R F^\epsilon_* R_\infty(-D) &\xrightarrow[\cong]{\cdot \frac{1}{g}} H_\m^d(R)\otimes_R R(K_R+D)\otimes_R F^\epsilon_*R_\infty (-D) \\
		& \to H_\m^d(R)\otimes_R R(K_R) \otimes_R F^\epsilon_* R_\infty((\pi^\epsilon-1)D).
	\end{align*}
	Take $c\in \adj_D(R,D,f^t)\cap R^{\circ,D}$. Then we have the following commutative diagram
	\[
		\xymatrix{
		H_\m^d(\omega_R) \ar[r]^-{\alpha} \ar[d] & H_\m^d(F^\epsilon_*R_\infty(-D)) \ar[d]^-{\beta} \ar[r] & H_\m^d(\omega_R)\otimes_R F^\epsilon_* R_\infty ((\pi^\epsilon-1)D) \ar[d] \\
		\ulim_p H_{\m_p}^d(\omega_{R_p}) \ar[r] & \ulim_p H_{\m_p}^d(F^{e_p}_*R_p(-D_p)) \ar[r]_-{\cong}^-{\ulim_p \frac{1}{g_p}} & \ulim_p (H_{\m_p}^d(\omega_{R_p})\otimes_{R_p} F^{e_p}_*R_p((p^{e_p}-1)D_p))
	},
	\]
	where $\alpha$ is defined as follows: Let $y\in \omega_R$, $s\in \N$ and $x$ be the product of a system of parameters $x_1, \dots, x_d$. Then 
	\[
		\alpha\left(\left[\frac{y}{x^s}\right]\right)=\left[\frac{F^\epsilon_*(cf^{\lceil t\pi^\epsilon\rceil}g^{\pi^\epsilon}y^{\pi^\epsilon})}{x^s}\right].
	\]
	In the above commutative diagram, $\beta$ is injective by \cite[Proposition 3.5, Proof of Propositoin 3.9]{Yam24}. Hence, $\eta\in 0_{H_\m^d(\omega_R)}^{*\mathrm{u}_DD,f^t}$, which completes the proof.
\end{proof}
\begin{thm} \label{divisorial ultra-test ideals = adjoint ideals}
	With notation as in Setting \ref{setting of rings with finitely generated anti-canonical algebras}, we have
	\[
		\tau_{D}^{\mathrm{u}}(R,D,f^t)=\adj_D(R,D,f^t).
	\]
\end{thm}
\begin{proof}
	By Proposition \ref{adj subset tau}, it is enough to show $\tau_{D}^{\mathrm{u}}(R,D,f^t)\subseteq\adj_D(R,D,f^t)$. By Proposition \ref{inifinite cyclic covers tau}, $\tau_D^{\mathrm{u}}(R,D,f^t)\subseteq \tau_{\mathscr{D}}^{\mathrm{u}}(\mathscr{R}_{\mathscr{M}},\mathscr{D},f^t)$. By Proposition $\ref{Q-Gorenstein case}$, we have
	\[
		\tau_{\mathscr{D}}^{\mathrm{u}}(\mathscr{R}_{\mathscr{M}},\mathscr{D},f^t)=\adj_{\mathscr{D}}(\mathscr{R}_{\mathscr{M}},\mathscr{D},f^t).
	\]
	Since $R_p\to (\mathscr{R}_{\mathscr{M}})_p$ is split injective for almost all $p$, by comparing with reductions modulo $p>0$ and applying the following lemma, we have $\adj_{\mathscr{D}}(\mathscr{R}_{\mathscr{M}},\mathscr{D},f^t)\cap R\subseteq \adj_D(R,D,f^t)$. Hence, the desired inclusion $\tau_{D}^{\mathrm{u}}(R,D,f^t)\subseteq\adj_D(R,D,f^t)$ follows.
\end{proof}
\begin{lem}
	$R\to S$ is a pure local homomorphism between $F$-finite normal local domains of characteristic $p>0$. Let $D$ and $E$ be prime divisors on $\Spec R$ and $\Spec S$, respectively. Suppose that $S(-E)=(R(-D)\cdot S)^{**}$. Let $f\in R^{\circ,D}$ and $t$ be a positive real number. Then we have
	\[
		\tau_E(S,E,f^t)\cap R\subseteq \tau_D(R,D,f^t).
	\]
\end{lem}
\begin{proof}
	Since $R\to S$ is pure, $\Spec S \to \Spec R$ is surjective. Since $S(-E)$ is the unique height one prime containing $R(-D)$, $S(-E)$ is lying over $R(-D)$. Therefore, $f\in S^{\circ, E}$.
	\begin{cl}
		For any $n\ge 0$, $R(nD)\subseteq S(nE)$.
	\end{cl}
	\begin{clproof}
		Let $g\in R(nD)$. Let $F$ be any prime divisor on $\Spec S$. If $F\neq E$, we can take $a\in R(-D)\setminus S(-F)$. Hence, $ga^n\in R\subseteq S$. Therefore, $\mathrm{ord}_F(\Div(g))\ge 0$. If $F=E$, then we can take $a\in R(-D)\setminus S(-2E)$ since $S(-E)=(R(-D)\cdot S)^{**}$. Then $ga^n\in R\subseteq S$. Hence, $\mathrm{ord}_{E}(\Div(g))\ge -n$. Therefore, we see that $g\in S(nE)$.
	\end{clproof}
	Hence, $R^{\circ,D}\subseteq S^{\circ,E}$ and $R((p^e-1)D)\subseteq S((p^e-1)E)$ for any $e\ge 0$. For any $R$-module $M$, we obtain
	\[
		0_{M}^{*_D,f^t}\hookrightarrow 0_{M\otimes_R S}^{*_EE,f^t}.
	\]
	By Proposition \ref{divisorial test ideal characterization}, we get
	\begin{align*}
		\tau_D(R,D,f^t) &= \bigcap_{M}\Ann_R 0_M^{*_DD,f^t} \\
			&\supseteq \bigcap_{M}\Ann_R 0_{M\otimes_R S}^{*_EE,f^t} \\
			&\supseteq \bigcap_{N}\Ann_R 0_{N}^{*_EE,f^t}\\
			&= \tau_E(S,E,f^t)\cap R,
	\end{align*}
	where $M$ (resp. $N$) runs through all $R$-modules (resp. $S$-modules).
\end{proof}
\begin{setting} \label{setting of pairs with ideals}
		Let $(R,\m)$ be a local normal domain essentially of finite type over $\C$ of dimension $d$, $D$ be a prime divisor or $D=0$,  $\ba\subseteq R$ be an ideal such that $\ba\cap R^{\circ,D}\neq \emptyset$ and $t$ be a positive rational number. Suppose that $\mathscr{R}=\bigoplus_{i\ge 0}R(-i(K_R+D))$ is a finitely generated $R$-algebra. Fix a canonical divisor $K_R$ of $\Spec R$ such that $K_R+D$ is effective and $K_R+D$ has no component equal to $D$.
\end{setting}
\begin{defn}
	With notation as in Setting \ref{setting of pairs with ideals}, take $f_1,\dots, f_n\in R^{\circ,D}$ such that $\ba=(f_1,\dots, f_n)$ and let $M$ be an $R$-module. Then we define
	\[
		0_{M}^{*\mathrm{u}_DD,f_1,\dots,f_n,t}:=\bigcap_{m\in \N}\bigcap_{i_1,\dots,i_n}0_{M}^{*\mathrm{u}_DD,(f_1^{i_1}\dots f_n^{i_n})^{\frac{1}{m}}},
	\]
	where $i_1,\dots,i_n$ run through all non-negative integers such that $i_1+\dots+i_n=\lceil mt \rceil$.
\end{defn}
\begin{thm} \label{adjoint ideals under pure morphisms}
	With notation as in Setting \ref{setting of pairs with ideals}, let $(S,\n)$ be a normal local domain essentially of finite type over $\C$. Suppose that $S$ is an $R$-algebra and $R\to S$ is a pure local $\C$-algebra homomorphism. Let $E$ be the Weil divisor on $\Spec S$ such that $S(-E)=(R(-D)\cdot S)^{**}$, where $(-)^{**}$ denotes the reflexive hull as an $S$-module, and suppose that $E$ is prime. Then we have
	\[
		\adj_E(S,E,(\ba S)^t)\cap R\subseteq \adj_D(R,D,\ba^t).
	\]
\end{thm}
\begin{proof}
	Take $f_1,\dots, f_n\in R^{\circ,D}$ such that $\ba=(f_1,\dots, f_n)$.
	Let $m, i_1,\dots,i_n\in \N$ such that $i_1+\dots+i_n=\lceil mt\rceil$, $f=f_1^{i_1}\dots f_n^{i_n}$ and $M$ be an $R$-module. 
	\begin{cl}
		$0_{M}^{*\mathrm{u}_DD,f^{1/m}}\hookrightarrow 0_{M\otimes_R S}^{*\mathrm{u}_E E,f^{1/m}}$
	\end{cl}
	\begin{clproof}
	For $\eta\in 0_{M}^{*\mathrm{u}_DD,f^{1/m}}$, there exists $c\in R^{\circ,D}$ such that for any $\epsilon\in {^*\N}$, $\eta\otimes F^\epsilon_*(cf^{\lceil \pi^{\epsilon}/m \rceil})=0$ in $M\otimes_R F^{\epsilon}_*R_\infty((\pi^\epsilon-1)D)$.
	By \cite[Proposition 6.7]{TY24}, we have a commutative diagram
	\[
	\xymatrix{
		M \ar@{^{(}->}[r]\ar[d]_-{\operatorname{id_M}\otimes F^\epsilon_*(cf^{\lceil \frac{1}{m}\pi^\epsilon \rceil})} & M\otimes_R S \ar[d]^-{\operatorname{id_M}\otimes F^\epsilon_*(cf^{\lceil \frac{1}{m}\pi^\epsilon \rceil})} \\
		M\otimes_R F^\epsilon_{*}R_\infty((\pi^\epsilon-1)D) \ar[r]& M\otimes_R F^\epsilon_*S((\pi^\epsilon-1)E).
	}
	\]
	Since $c\in R^{\circ,D}\subseteq S^{\circ,E}$ by \cite[Proposition 6.6]{TY24}, we have $\eta\in 0_{M\otimes_R S}^{*\mathrm{u}_E E,f^{1/m}}$.	
	\end{clproof}
	By the above claim, it follows that
	\begin{align*}
		0_M^{*\mathrm{u}_D D,f_1,\dots,f_n,t} \subseteq 0_{M\otimes_R S}^{*\mathrm{u}_E E,f_1,\dots,f_n,t}.
	\end{align*}
	By an argument similar to Proposition \ref{tau over all modules}, we have
	\[
		\Ann_{R}0_{H_\m^d(\omega_R)}^{*\mathrm{u}_D D,f_1,\dots,f_n,t}=\bigcap_{M}\Ann_{R}0_M^{*\mathrm{u}_D D,f_1,\dots,f_n,t},
	\]
	and
	\[
	\Ann_{S}0_{H_\n^{\dim S}(\omega_S)}^{*\mathrm{u}_E E,f_1,\dots,f_n,t}=\bigcap_{N}\Ann_{S}0_N^{*\mathrm{u}_E E,f_1,\dots,f_n,t},
	\]
	where $M$ (resp. $N$) runs thorough all $R$-modules (resp. $S$-modules). On the other hand, we have
	\begin{align*}
		\Ann_{R}0_{H_\m^d(\omega_R)}^{*\mathrm{u}_D D,f_1,\dots,f_n,t}&=\Ann_R\left(\bigcap_{m\in \N}\bigcap_{i_1,\dots,i_n}0_{H_\m^d(\omega_R)}^{*\mathrm{u}_DD,(f_1^{i_1}\dots f_n^{i_n})^{\frac{1}{m}}}\right)\\
		&=\Ann_R \Ann_{H_\m^d(\omega_R)} \left(\sum_{m\in \N}\sum_{i_1,\dots,i_n}\adj_D(R,D,(f_1^{i_1}\dots f_n^{i_n})^{\frac{1}{m}})\right)\\
		&= \Ann_R \Ann_{H_\m^d(\omega_R)} \adj_D(R,D,\ba^t) \\
		&= \adj_D(R,D,\ba^t),
	\end{align*}
	where $i_1,\dots,i_n$ run through all non-negative integers such that $i_1+\dots+i_n=\lceil mt \rceil$ and the third equality follows from \cite[Proof of Lemma 8.3]{TY24}.
	Similarly, we obtain $\adj_E(S,E,(\ba S)^t)=\Ann_{S}0_{H_\n^{\dim S}(\omega_S)}^{*\mathrm{u}_E E,f_1,\dots,f_n,t}$. Hence, we have
	\begin{align*}
		\adj_D(R,D,\ba^t) &= \Ann_{R}0_{H_\m^d(\omega_R)}^{*\mathrm{u}_D D,f_1,\dots,f_n,t} \\
		& = \bigcap_{M}\Ann_{R}0_M^{*\mathrm{u}_D D,f_1,\dots,f_n,t} \\
		& \supseteq \bigcap_{M\otimes_R S}\Ann_{R}0_{M\otimes S}^{*\mathrm{u}_E E,f_1,\dots,f_n,t}\\
		& \supseteq \left(\bigcap_{N}\Ann_{S}0_N^{*\mathrm{u}_E E,f_1,\dots,f_n,t}\right)\cap R\\
		& = \left(\Ann_{S}0_{H_\n^{\dim S}(\omega_S)}^{*\mathrm{u}_E E,f_1,\dots,f_n,t}\right) \cap R\\
		& =\adj_E(S,E,(\ba S)^t)\cap R,
	\end{align*}
	where $M$ (resp. $N$) runs thorough all $R$-modules (resp. $S$-modules).
\end{proof}
\begin{rem}
	The above theorem can be extended to the case when $t$ is a positive real number because 
	\[
		\adj_E(S,E,(\ba S)^{t+\epsilon})=\adj_E(S,E,(\ba S)^t)
	\]
	for any $0<\epsilon \ll 1$.
\end{rem}
\begin{ques} \label{question}
	Can the above theorem be extended to triples $(R,D+\Gamma,\ba^t)$, where $\Gamma$ is an effective $\Q$-Weil divisor which has no component equal to $D$, its anti-log-canonical ring is finitely generated, and $\ba$ and $t$ as in Setting \ref{setting of rings with finitely generated anti-canonical algebras}?
\end{ques}
 Question \ref{question} has an affirmative answer if $K_X+D+\Gamma$ is $\Q$-Cartier (see \cite{TY24}). 
As an application of the main theorem, we show the behavior of singularities of lc type under pure morphisms.
\begin{rem}
	When $X=\Spec A$, $Y=\Spec B$ are both affine, then $Y\to X$ is pure if and only if $A\to B$ is pure (\cite[Lemma 2.2]{HH95}).
\end{rem}
\begin{lem}
	Let $X$ be a normal affine variety over an algebraically closed field of characteristic zero. Suppose that $\mathscr{R}=\bigoplus_{i\ge 0}\sO_X(-iK_X)$ is a finitely generated $\sO_X$-algebra. Take a nonzero element $f$ of the multiplier ideal $\mathcal{J}(X)$. $X$ is of lc type if and only if for any $0<\epsilon<1$, $f\in \mathcal{J}(X,(1-\epsilon)\Div (f))$.
\end{lem}
\begin{proof}
	The ``only if'' direction follows from \cite[Lemma 8.12]{TY24}. Let $\pi:X'=\Proj \mathscr{R}\to X$. By \cite[Corollary 2.25]{CEMS18}, we see that $\pi_*\mathcal{J}(X',\pi^*\Delta)=\mathcal{J}(X,\Delta)$ for any $\Q$-Cartier divisor $\Delta$ on $X$. Suppose that $f\in \mathcal{J}(X,(1-\epsilon)\Div(f))$ for any $0<\epsilon<1$. Then $f\in H^0(X',\mathcal{J}(X',(1-\epsilon)\Div (f)))$. Since $X'$ is $\Q$-Gorenstein, $X'$ has log canonical singularities by \cite[Lemma 8.12]{TY24}. An argument similar to \cite[Lemma 2.4]{Zhuang24} implies that $X$ is of lc type.
\end{proof}
\begin{thm} \label{log canonical under pure morphisms}
		Let $f:Y\to X$ be a pure morphism between normal complex varieties and suppose that $\bigoplus_{i\ge 0} \sO_X(-iK_X)$ is finitely generated.
		Assume in addition that one of the following conditions holds. 
		\begin{enumerate}
			\item There exists an effective $\Q$-Weil divisor $\Delta$ on $Y$ such that $K_Y+\Delta$ is $\Q$-Cartier and no non-klt center of $(Y, \Delta)$ dominates $X$. 
			\item The non-klt-type locus of $Y$ has dimension at most one. 
		\end{enumerate}
		If $Y$ is of lc type, then $X$ is of lc type.
\end{thm}
\begin{proof}
	This follows from \cite[Proof of Theorem 8.13]{TY24}. We use Theorem \ref{adjoint ideals under pure morphisms} in stead of \cite[Theorem 1.2]{Yam23}.
\end{proof}
\bibliography{bibtex.bib}
\bibliographystyle{alpha}

\bigskip

\end{document}